\documentclass[runningheads,envcountsect]{llncs}
\usepackage{graphicx}
\usepackage{amsmath,amssymb,textcomp}
\usepackage{hyperref}

\newcommand{\Inf}{\textnormal{\textbf{Inf}}}
\newcommand{\Ex}{\mathbf{Ex}}
\newcommand{\infi}{\mathrm{inf}}

\newcommand{\skipmm}[1]{\hspace{-#1mm}}

\DeclareMathOperator{\A}{\mathcal{A}}
\DeclareMathOperator{\type}{type}

\begin{document}
\title{On the Turing complexity of learning finite families of algebraic structures\thanks{Bazhenov was supported by the Mathematical Center in Akademgorodok under agreement No.~075-15-2019-1613 with the Ministry of Science and Higher Education of the Russian Federation.
San Mauro was supported by the Austrian Science Fund FWF, project M 2461.}}
\titlerunning{Complexity of Learning Finite Families of Structures}
\author{Nikolay Bazhenov\inst{1}\orcidID{0000-0002-5834-2770} \and
Luca~San~Mauro\inst{2}\orcidID{0000-0002-3156-6870}
}

\authorrunning{N.~Bazhenov and L.~San Mauro}
\institute{Sobolev Institute of Mathematics, 4 Acad. Koptyug Ave., Novosibirsk,\\
630090, Russia\\
\email{bazhenov@math.nsc.ru} \\
\and
Institute of Discrete Mathematics and Geometry, Vienna University of Technology
\email{luca.sanmauro@gmail.com}}
\maketitle              % typeset the header of the contribution
\begin{abstract}
In previous work, we have combined computable structure theory and algorithmic learning theory to study which families of algebraic structures are learnable in the limit (up to isomorphism). In this paper, we measure the computational power that is needed to learn finite families of structures. In particular, we prove that, if a  family of structures is both finite and learnable, then any oracle which computes the Halting set is able to achieve such a learning. On the other hand, we construct a pair of structures which is learnable but no computable learner can learn it.

\end{abstract}

\section{Introduction}

Algorithmic learning theory, introduced by Gold~\cite{Gold67} in the 1960's,  comprises different formal models for
the inductive inference.  Broadly construed, this research program   deals with the question of how a
\emph{learner}, provided with more and more data about some
\emph{environment}, is eventually able to achieve systematic
knowledge about it (see \cite{Osh-Sto-Wei:b:86:stl} for an introduction to this area).
Most work in algorithmic learning theory concerns either  learning of recursive functions~\cite{zz-tcs-08} or learning of formal languages~\cite{lange2008learning}. These paradigms model the data to be learned as an
\emph{unstructured flow}. 
Yet, researchers have been also considering  (although less systematically) the learning of data embodied with a structural content, often focusing on special classes of algebraic structures: e.g., in a series of papers, Stephan and his collaborators considered commutative
rings~\cite{SV01}, trees~\cite{MS04}, vector spaces~\cite{HaSt07}, and matroids~\cite{GaoStephan-12}.

A broader pursuit is to develop a framework which can  be applied to arbitrary structures. This has been done first by Glymour~\cite{Gly:j:85} and later expanded by Martin and Osherson~\cite{Mar-Osh:b:98}. We contributed to this research thread with our own paradigm~\cite{FKS-19,bazhenov2020learning}, which is inspired by ideas and techniques from computable structure theory. Intuitively (but see below for formal details), we say that a countably infinite structure $\mathcal{A}$ is \emph{learnable} if one can eventually guess its isomorphism type by seeing larger and larger (but always finite) substructures of $\mathcal{A}$. As for classical paradigms, the emphasis is not on learning single structures but rather families of structures (the former case being trivial). 

In \cite{bazhenov2020learning}, we adopted infinitary logic to obtain a model-theoretic characterization of which families of structures are learnable. Remarkably, there are families of structures which are finite, up to isomorphism, and \emph{not} learnable; this contrasts with classical paradigms, since, e.g., any finite collection of recursive functions \emph{is} learnable. 

In this paper, we advance the knowledge about the learnability of finite families of structures. Specifically, by relying on Turing complexity we evaluate the computational power that is needed  to learn such families. We prove that there is a pair of structures which is learnable but no computable learner can learn it, solving a question left open in \cite{bazhenov2020learning}. 

\section{Preliminaries}  We assume that the reader is familiar with the basic notions of classical  computability theory; in any case,  our terminology and notations are standard and as in \cite{soare2016turing}.  In
particular, by $\{\varphi_e\}_{e\in\omega}$, $\{W_e\}_{e\in\omega}$, and $\{\Phi^X_e\}_{e\in\omega}$ we  denote a  uniformly computable list of, respectively, all partial computable functions, all computably enumerable (c.e.) sets, and all Turing
operators with oracle $X$.

\subsubsection*{Computable structures.}  A \emph{signature} $L$ lists all function symbols and relation symbols which characterize an algebraic structure. In this paper, we consider only \emph{relational signatures}, i.e., signatures with no function symbols. Furthermore, all our structures have  domain the set $\omega$ of the natural numbers. We say that two structures are \emph{copies} of each other if they are isomorphic. 
In computable structure theory, one measures the complexity of an $L$-structure $\mathcal{A}$ by identifying $\mathcal{A}$ with its \emph{atomic diagram},  i.e., the collection of atomic formulas  which are true of $\mathcal{A}$. Up to a suitable G\"odel numbering of $L$-formulas, the atomic diagram of $\mathcal{A}$ may be regarded as a subset of $\omega$: this provides a natural way of assigning to each structure a Turing degree $\mathbf{d}$, representing its algorithmic complexity. 
Any computable structure $\mathcal{A}$ in a relational signature can be presented
as an increasing union of its finite substructures
\[
\A\restriction_0 \ \subseteq \A\restriction_1 \ \subseteq  \ldots \subseteq \A\restriction_i \ \subseteq \ldots,
\]
where $\A\restriction_n$ denotes the restriction of $\A$ to the domain $\{0,1,\ldots,n\}$  and $\A=\bigcup \mathcal{A} \restriction_i$.  For more background about computable structures, see \cite{AK00,EG-00}.

\subsubsection*{Infinitary formulas.}
To assess the model-theoretic complexity of countable structures, it is common to work in the infinitary logic $\mathcal{L}_{\omega_1\omega}$, which allows to take the conjunctions or disjunctions of infinite sets of formulas. In particular, \emph{infinitary $\Sigma_n$ formulas} are defined as follows,
\smallskip

\begin{itemize}
	\item $\Sigma^{\inf}_0$  and $\Pi^{\inf}_0$ formulas are quantifier-free first-order formulas.
	
	\item A $\Sigma_{n+1}^{\inf}$ formula $\psi(\bar{x})$ is a countably infinite disjunction
	\[
		\underset{i\in I}{\bigvee\skipmm{5.5}\bigvee} \exists \bar y_i \xi_i(\bar x, \bar y_i),
	\]
	where each $\xi_i$ is a $\Pi^{\inf}_{n}$ formula.
	
	\item A $\Pi_{n+1}^{\inf}$ formula $\psi(\bar x)$ is a countably infinite conjunction
	\[
		\underset{i\in I}{\bigwedge\skipmm{5.5}\bigwedge} \forall \bar y_i \xi_i(\bar x, \bar y_i),
	\]
	where each $\xi_i$ is a $\Sigma^{\inf}_{n}$ formula.
\end{itemize}
Next, \emph{computable infinitary $\Sigma_n$ formulas} (or \emph{$\Sigma^c_n$ formulas}, for short) are defined in the same way as above, but requiring infinite conjunctions and disjunctions to range over c.e.\ sets of  (computable) formulas. Finally, computable infinitary formulas can be relativized to an arbitrary oracle $X$: the class of \emph{$X$-computable infinitary $\Sigma_n$~formulas} is denoted by $\Sigma^c_n(X)$. For more background about infinitary formulas, see~\cite{marker2016lectures}.

\subsection{Our framework (for  finite families)} We shall now revisit $\Inf\mathbf{Ex}_{\cong}$, 
the learning paradigm  presented in \cite{bazhenov2020learning}. Here, the exposition is somehow simplified by the fact that we will focus only on finite families of structures: this allows us to ignore how a given family is enumerated, which was a crucial source of complexity in \cite{bazhenov2020learning}. For current purposes, it is in fact enough to assume that any structure $\A$ gives rise to a corresponding   \emph{conjecture} $\ulcorner \A\urcorner$, to be understood as conveying the piece of information ``this is $\A$''. 

Suppose that $\mathbf{P}$ is the learning problem associated to  a finite family $\mathfrak{K}$ of (non-isomorphic) computable structures. The ingredients of our framework may be specified as follows. For $\mathbf{P}$,

\begin{itemize}
\item The \emph{learning domain} ($\mathrm{LD}$)  is the collection of all copies of the structures from $\mathfrak{K}$. That is,
\[
\mathrm{LD}(\mathfrak{K}):=\bigcup_{\A \in \mathfrak{K}} \{\mathcal{S} : \mathcal{S}\cong \A\}.
\]
\item The \emph{hypothesis space} ($\mathrm{HS}$) contains, for each $\A\in \mathfrak{K}$, a formal symbol $\ulcorner \A \urcorner$ (these symbols serve as conjectures about the isomorphism type of the observed structure) and a question mark symbol. That is, 
\[
\mathrm{HS}(\mathfrak{K}):=\{\ulcorner \A \urcorner : \A \in  \mathfrak{K} \}\cup \{?\}.
\]
\item A \emph{learner} $M$ sees, by stages, all positive and negative data about any given structure in the learning domain and is required to output conjectures. This is formalized by saying that $M$ is a function
\[
\mbox{from }\{\mathcal{S}\restriction_n \ \colon \mathcal{S}\in LD(\mathfrak{K}) \} \mbox{ to } \mathrm{HS}(\mathfrak{K}).
\]
\item The learning is \emph{successful} if, for each structure $\mathcal{S}\in\mathfrak{K}$, the learner eventually
stabilizes to a correct conjecture about its isomorphism type. That is, 
\[
\lim_{n\to \infty} M(\mathcal{S}\restriction_n)=\ulcorner \mathcal{A}\urcorner  \mbox{ if and only if $\mathcal{S}$ is a copy of $\A$}.
\]
We say that $\mathfrak{K}$ is \emph{learnable}, if some learner $M$ successfully learns $\mathfrak{K}$.
 \end{itemize}

%Let us now provide the formal details of our framework. 
%
%
%Let $\mathfrak{K}$ be a finite family of computable structures. Define the following sets,
%\[
%\mathrm{Sub}(\mathfrak{K}):= \bigcup_{\A\in\mathfrak{K}}\{\mathcal{B}\restriction_n : \B \cong \A, n\in\omega \},
%\,
%\mathrm{Conj}(\mathfrak{K}):=\{\ulcorner \A \urcorner : \A \in  \mathfrak{K} \}\cup \{? \}.
%\]
% A learner $M$ is a function from $\mathrm{Sub}(\mathfrak{K})$ to $\mathrm{Conj}(\mathfrak{K})$.
%
%A learner $M$ is a function mapping all $A\restriction_n$'s of a given structure $A$ to
%the set of conjectures of $\mathfrak{K}$.
%
%$\mathfrak{K}$ is $\Inf\Ex_{\cong}$-\emph{learnable} if there is a function $M$ (called \emph{a learner}
%\[
%M : \{\mathcal{A}\restriction_n : n\in\omega \} \to \{ \ulcorner \mathcal{B} \urcorner : \mathcal{B}  \in \mathfrak{K} \}
%\]
%such that, for any copy $\mathcal{S}$ of $\mathcal{A}$,
%\[
%\lim_{n\mapsto\infty} M(\mathcal{S}\restriction_n)=\ulcorner \mathcal{A}\urcorner.
%\]
%That is, in the limit L learns any computable structure from K.

\begin{remark}
The interested reader is referred to  \cite{bazhenov2020learning} for motivating examples and a detailed discussion about our framework. Note that in that paper, inspired by an established notation in algorithmic learning theory, we named our paradigm  \emph{$\Inf\Ex_{\cong}$-learning}. Here, since there is no risk of ambiguity, we just say that some family $\mathfrak{K}$ is learnable. 
\end{remark}

 In \cite{bazhenov2020learning}, we showed that asking if a family (possibly infinite) is learna\-ble is the same as asking whether the structures from $\mathfrak{K}$ can be distinguished by $\Sigma^{\inf}_2$ formulas. The next theorem is an immediate consequence of Theorem~3.1 and Corollary~4.1 of~\cite{bazhenov2020learning}.

\begin{theorem}[Bazhenov, Fokina, San Mauro] \label{theorem:characterization learning}
Let $\mathfrak{K}$ be a finite family of pairwise nonisomorphic structures $\{\A_0, \ldots, \A_n\}$. Then,
\begin{enumerate}
\item[(a)] $\mathfrak{K}$ is learnable  if and only if there are $\Sigma^{\inf}_2$ formulas $\phi_0,\ldots,\phi_n$ such that
\[
\A_i \models \phi_j \Leftrightarrow i=j.
\]
\item[(b)] $\mathfrak{K}$ is learnable via an $X$-computable learner  if and only if there are $\Sigma^c_2(X)$ formulas $\psi_0,\ldots,\psi_n$ such that
\[
\A_i \models \psi_j \Leftrightarrow i=j.
\]
\end{enumerate}
\end{theorem}

\section{An upper bound to the learners' complexity}

According to our framework, a learner $M$ can be a function of any complexity. So, it is natural to ask how learnability is affected if we restrict our focus to learners of some bounded Turing complexity. More precisely, in this paper we want to understand, for a learnable family $\mathfrak{K}$, how powerful an oracle should be  to achieve such learning. The next theorem states that, if $\mathfrak{K}$ is finite, then $\mathbf{0}'$ suffices. Similar problems for classical learning paradigms have been intensively studied: e.g., Kummer and Stephan~\cite{kummer1996structure} showed, the whole of class of c.e.\
languages is learnable from informant relative to oracle $A$ if and only if the Halting problem
is Turing reducible to $A$ (for more results about learning with oracles, see  \cite{adleman1991inductive,slaman1991oracles,kummer1996structure}). 

\begin{theorem}\label{theo:upper-bound-finite-case}
	Let $\mathfrak{K}$ be a finite family of computable structures. If $\mathfrak{K}$ is learnable, then it is learnable by a $\mathbf{0}'$-computable learner.
\end{theorem}
\begin{proof}
	Suppose that $\mathfrak{K}$ is equal to $\{\mathcal{A}_0, \ldots, \mathcal{A}_n\}$. By item $(a)$ of Theorem~\ref{theorem:characterization learning}, there are $\Sigma^{\infi}_2$ sentences $\phi_0,\dots,\phi_n$ such that for all $i,j\leq n$,
	\begin{equation}\label{equ:property}
		\mathcal{A}_j \models \phi_i \ \Leftrightarrow\ j=i.
	\end{equation}
	
	We show that one can replace formulas $\phi_i$ with $\Sigma^c_2(\mathbf{0}')$ sentences $\psi_i$, while preserving Eq.~(\ref{equ:property}). Then item $(b)$ of Theorem \ref{theorem:characterization learning} will imply that the class $\mathfrak{K}$ is learnable by a $\mathbf{0}'$-computable learner. We describe the construction of $\psi_0$~--- the remaining $\psi_i$ can be recovered in a similar way.
	
	Suppose that $\phi_0$ is equal to
	\[
		\underset{i\in I}{\bigvee\skipmm{5.5}\bigvee} \exists \bar x_i \xi_i(\bar x_i),
	\]
	where every $\xi_i$ is a $\Pi^{\infi}_1$~formula. Without loss of generality, one may assume that $\mathcal{A}_0 \models \exists \bar x_0 \xi_0(\bar x_0)$. By Eq.~(\ref{equ:property}), we have 
	\begin{equation} \label{equ:auxil-upper}
		\mathcal{A}_{\ell} \not\models \exists \bar x_0 \xi_0(\bar x_0) \text{ for all } \ell\neq 0.
	\end{equation}
	
	Suppose
	\[
		\xi_0(\bar x_0) = \underset{j\in J}{\bigwedge\skipmm{5.5}\bigwedge} \forall \bar y_j \theta_j(\bar x_0, \bar y_j),
	\]
	where every $\theta_j$ is quantifier-free. 
	
	Choose a tuple $\bar c$ from $\mathcal{A}_0$ such that $\mathcal{A}_0 \models \xi_0(\bar c)$. Consider a set of formulas
	\begin{gather*}
		\type_{\forall}(\bar c) = \{ \chi(\bar x_0) \,\colon \chi \text{ is a finitary } \forall \text{-formula},\ \mathcal{A}_0 \models \chi(\bar c) \}.
	\end{gather*}
	Since $\mathcal{A}_0$ is a computable structure, the set $\type_{\forall}(\bar c)$ is co-c.e. Furthermore, $\{ \forall \bar y_j \theta_j(\bar x_0, \bar y_j)\,\colon j\in J\} \subseteq \type_{\forall}(\bar c)$. 
	\smallskip
	
	Consider a $\Sigma^c_2(\mathbf{0}')$ sentence
	\[
		\psi_0 := \exists \bar x_0 \underset{\chi \in type_{\forall}(\bar c)}{\bigwedge\skipmm{5.5}\bigwedge} \chi(\bar x_0).
	\]
	By employing Eq.~(\ref{equ:auxil-upper}), it is easy to show that $\mathcal{A}_0 \models \psi_0$ and $\mathcal{A}_{\ell} \not\models \psi_0$ for every $\ell> 0$. In other words, the formula $\psi_0$ satisfies the desired properties. We deduce that  $\mathfrak{K}$ is learnable by a $\mathbf{0}'$-computable learner.
	\qed
\end{proof}

The following question, which naturally originates from the theorem above, is left open.

\begin{question}
Is there an intermediate degree $\mathbf{0}<\mathbf{d}<\mathbf{0}'$ such that, if a finite family $\mathfrak{K}$ is learnable, then it is learnable by a $\mathbf{d}$-computable learner?
\end{question}

\section{A lower bound to the learners' complexity}
Having shown that an oracle as weak as $\mathbf{0}'$ is able to retrieve the learning process for any finite family which is learnable, one may ask whether, for finite families, learnability even coincides with \emph{computable} learnability. In this section, we show that this is not the case, by constructing a pair of structures which, although learnable, cannot be learned by any computable learner. This answers the question left open in \cite[p.~17]{bazhenov2020learning}. Our construction is based on a family of structures built by Alaev~\cite{Alaev-03}.

%(see Lemma \ref{lem:Alaev-01} below).

\begin{theorem}\label{theo:lower-bound-finite-case}
	There is a pair $\mathfrak{K}$ of computable structures  such that $\mathfrak{K}$ is learnable, but not learnable by a computable learner.
\end{theorem}

\begin{proof}
For the sake of exposition, the proof is split into several subsections. 

\subsection*{Preliminaries of the construction}

We treat trees as undirected graphs. Consider a relational signature $L_0 := \{ R^2, Q^1, U^1\}$. We define an auxiliary computable $L_0$-structure $\mathcal{V}$ as follows.
\begin{itemize}
	\item The domain of $\mathcal{V}$ is equal to
	\[
		\{r, c\} \cup \{ a_i, b_i\,\colon i\in \omega\} \cup \{ d_{i,j}\,\colon i \in\omega,\, j\leq i\}.
	\]
	
	\item $Q^{\mathcal{V}} = \{ c\}$ and $U^{\mathcal{V}} = \{ r\} \cup \{ a_i, b_i\,\colon i\in \omega\}$.
	
	\item The structure $(\mathrm{dom}(\mathcal{V}),R^{\mathcal{V}})$ is an undirected graph, which contains the following edges:
	\begin{itemize}
		\item $(r,a_i)$  and $(a_i,b_i)$ for all $i\in\omega$;
		
		\item $(a_i,d_{i,0})$, $(d_{i,0}, d_{i,1})$, $(d_{i,1}, d_{i,2})$, \dots, $(d_{i,i-1}, d_{i,i})$, $(d_{i,i},c)$.
	\end{itemize}
\end{itemize}

The informal idea behind the structure $\mathcal{V}$ is as follows. If we consider its substructure $\mathcal{S}$ on the domain $U^{\mathcal{V}}$, then $\mathcal{S}$ is a tree with root $r$ and infinitely many branches of size 2. 

For a given $i\in\omega$, the elements $d_{i,j}$, $j\leq i$, serve to distinguish $a_i$ in a first-order way. Let $\theta_i(x,y_0,\dots,y_i,z)$ be a quantifier-free formula, which is built as the conjunction of the following facts:
\begin{itemize}
	\item all elements $x,y_0,\dots,y_i,z$ are pairwise distinct;
	
	\item $U(x)$ and $Q(z)$;
	
	\item $\neg U(y_j)$ for all $j\leq i$;
	
	\item the sequence $x,y_0,\dots,y_i,z$ forms a chain with respect to the graph relation $R$.
\end{itemize}
Then we have:
\begin{equation}\label{equ:001-theta}
	\mathcal{V} \models (x = a_i) \leftrightarrow \exists y_0 \dots \exists y_i\exists z \theta_i(x,\bar y, z).
\end{equation}

\medskip

We use the structure $\mathcal{V}$ to code c.e.\ sets. For a c.e. set $W$, by $\mathcal{T}[W]$ we denote the substructure of $\mathcal{V}$ on the c.e. domain
\[
	\{r, c\} \cup \{ a_i \,\colon i\in \omega\} \cup \{ d_{i,j}\,\colon i \in\omega,\, j\leq i\} \cup \{b_i\,\colon i \in W\}.
\]
See Fig.~\ref{fig:tw} for an example of this encoding.
Without loss of generality, one may assume that the structure $\mathcal{T}[W]$ is computable. Furthermore, given an index $e\in\omega$, one can effectively find a computable index of the structure $\mathcal{T}[W_e]$.

\begin{figure}
\begin{center}
\includegraphics[scale=.8]{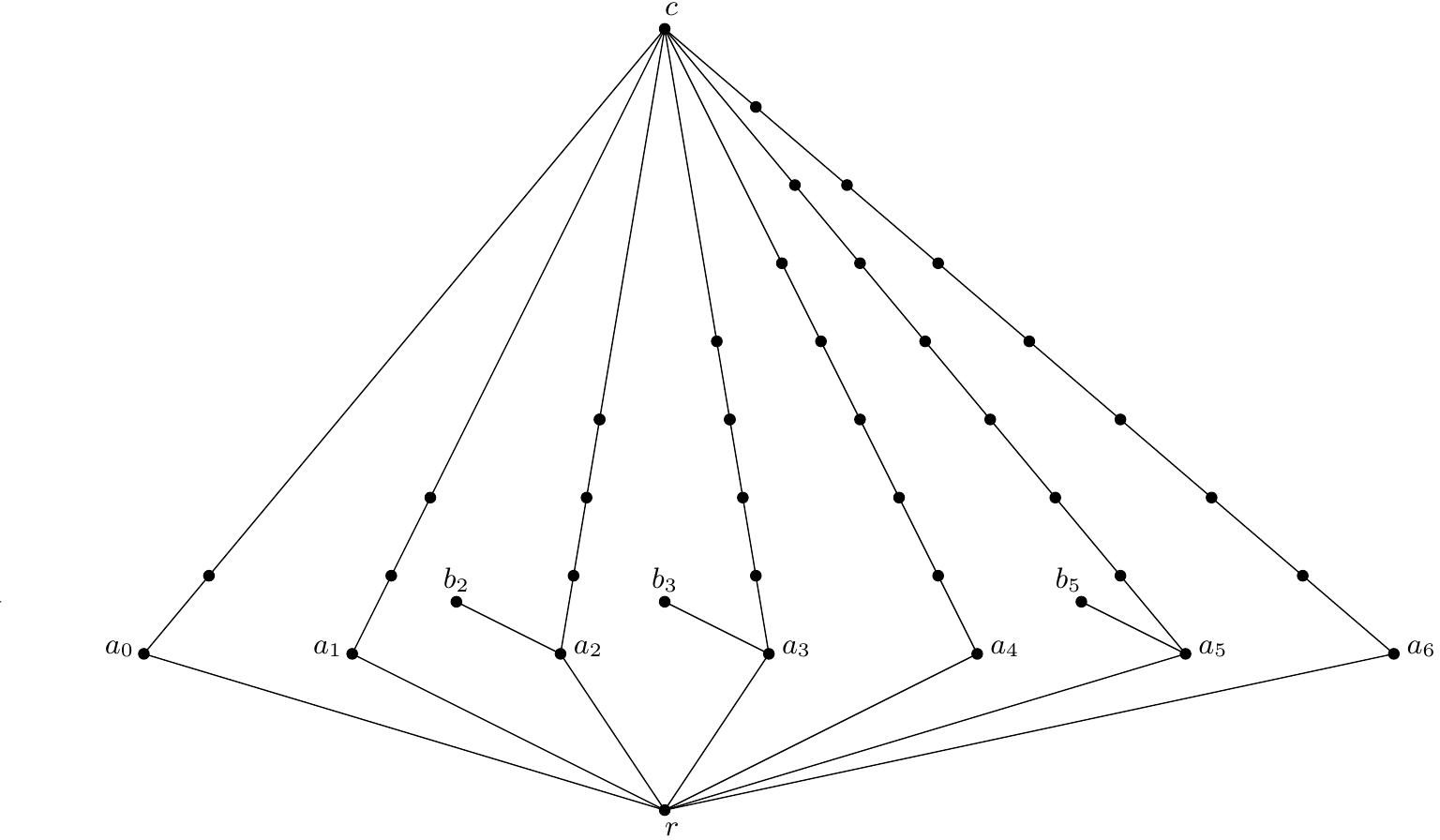}
\end{center}
\caption{A fragment of $\mathcal{T}[P]$, where $P$ is the set of all primes.}
\label{fig:tw}
\end{figure}

\smallskip

We will use the following technical lemmas:

\begin{lemma}[folklore] \label{lem:folklore} 
The set $\{ e\in\omega \,\colon W_e \subseteq \emptyset'\}$ is $\Pi^0_2$-complete.
\end{lemma}

\begin{lemma}[see Lemma~5 of~\cite{Alaev-03}] \label{lem:Alaev-01}
	Let $\mathcal{I}$ be a family, which consists of all c.e. sets $W$ such that $\emptyset' \subseteq W$ and $\mathrm{card} (W \setminus \emptyset') = 1$. The family $\mathcal{I}$ has the following properties.
	\begin{enumerate}
		\item $\mathcal{I}$ is uniformly enumerable, i.e. there is a computable function $h(x)$ such that the family $\{ W_{h(e)} \,\colon {e\in\omega}\}$ equals $\mathcal{I}$. Note that in general, the sequence $(W_{h(e)})_{e\in\omega}$ \emph{allows repetitions}.
		
		\item There is a computable function $f_{\mathcal I}(x)$ so that for every $x$, $W_{f_{\mathcal{I}}(x)} = \emptyset'\cup \{x\}$. In particular, for all $e\in \omega$:
		\begin{itemize}
			\item[(2.a)] If $W_e \subseteq \emptyset'$, then $W_{f_{\mathcal{I}}(e)} = \emptyset'$.
			
			\item[(2.b)] If $W_e \not\subseteq \emptyset'$, then $W_{f_{\mathcal{I}}(e)} \in \mathcal{I}$.
		\end{itemize}
	\end{enumerate}
\end{lemma}

%%%%%%%%%%%%%%%%%

\subsection*{Construction}

Choose an arbitrary set $X\in \Delta^0_3 \smallsetminus \Delta^0_2$.

Fix two binary $\Pi^0_2$ predicates $P_{+}$ and $P_{-}$ such that for any $k\in\omega$,
\begin{equation}\label{equ:002-predicate-P}
	\begin{array}{l}
		k \in X \ \Leftrightarrow\ \exists x P_{+}(k,x);\\
		k \not\in X \ \Leftrightarrow\ \exists x P_{-}(k,x).
	\end{array}
\end{equation}

Let $\circ \in \{ +,-\}$. By Lemma~\ref{lem:folklore}, there is a computable function $g_{\circ}(k,x)$ such that $P_{\circ}(k,x)$ is equivalent to the condition $W_{g_{\circ}(k,x)} \subseteq \emptyset'$. Hence, by Lemma~\ref{lem:Alaev-01}, we have:
\begin{itemize}
	\item If $P_{\circ}(k,x)$ is true, then $W_{f_{\mathcal{I}}(g_{\circ}(k,x))} = \emptyset'$.
	
	\item If $P_{\circ}(k,x)$ is false, then $W_{f_{\mathcal{I}}(g_{\circ}(k,x))} \in \mathcal{I}$.
\end{itemize}

\medskip

Consider a signature $L := L_0 \cup \{ T^1; a\}$, where $a$ is a constant symbol. For a number $k\in\omega$, we define a computable $L$-structure $\mathcal{C}_k$ as follows. If it is not specified otherwise, we assume that a freshly added element $x$ of $\mathcal{C}_k$ satisfies $\neg T(x)$.
\begin{enumerate}
	\item[(A)] Choose a fresh element $r^k$ as the root of the tree. Put $a^{\mathcal{C}_k} := r^k$.
	
	\item[(B)] For each set $W \in \mathcal{I}$, append to $r^k$ infinitely many copies $\mathcal{D}_{\ell}$, $\ell \in\omega$, of the structure $\mathcal{T}[W]$. For each of $\mathcal{D}_{2 m}$, $m\in\omega$, we set $\mathcal{C}_k \models T(q)$, where $q$ is the root of $\mathcal{D}_{2 m}$.
	
	\item[(C)] For each $\circ\in\{ +, -\}$ and each $m\in\omega$, append to $r^k$ infinitely many copies $\mathcal{E}_{\ell}$, $\ell\in\omega$, of the structure $\mathcal{T}[W_{f_{\mathcal{I}}(g_{\circ}(k,m))}]$. Let $q$ be the root of $\mathcal{E}_{\ell}$. If $\circ$ equals $+$, then set $\mathcal{C}_k \models T(q)$.
\end{enumerate}
Since the family $\mathcal{I}$ is uniformly enumerable, it is not hard to show that the structures $\mathcal{C}_k$ are uniformly computable.

Let $\mathcal{C}_{base}$ be the $L$-structure obtained by employing only the steps~(A) and~(B) described above. From~(\ref{equ:002-predicate-P}), it is not hard to establish the following:
\begin{enumerate}
	\item If $k \in X$, then  $\mathcal{C}_k$ is isomorphic to the structure $\mathcal{C}^{+}$, which is constructed by appending to $\mathcal{C}_{base}$ infinitely many copies of $\mathcal{T}[\emptyset']$ with their roots satisfying the predicate $T$.
	
	\item If $k\not\in X$, then $\mathcal{C}_k$ is isomorphic to the structure $\mathcal{C}^{-}$, which is obtained by appending to $\mathcal{C}_{base}$ infinitely many copies of $\mathcal{T}[\emptyset']$ with their roots satisfying the formula $\neg T(x)$.
\end{enumerate}

%%%%%%%%%%%%%%%%%

\subsection*{Verification}

Consider class $\mathfrak{K}$, containing the isomorphism types of the structures $\mathcal{C}^{+}$ and $\mathcal{C}^{-}$. We prove that the class $\mathfrak{K}$ satisfies our theorem.

We define the following $\Sigma^{\infi}_{2}$-sentences: for $\circ \in \{ +, -\}$,
\begin{multline*}
	\xi_{\circ} := \exists x \big[ R(a,x) \,\&\, T^{\circ}(x) \,\& \\ 
	\underset{i\in \overline{\emptyset'}}{\bigwedge\skipmm{5.5}\bigwedge} \forall \widehat{x} \forall y_0 \dots \forall y_i \forall z \big[ R(x,\widehat{x}) \,\&\, \theta_i(\widehat{x},y_0,\dots,y_i,z) \rightarrow\\ 
	 \neg \exists v ( R(\widehat{x},v) \,\&\, U(v) \,\&\, v\neq x )  \big]\big],
\end{multline*}
where $\theta_i$ are formulas from Eq.~(\ref{equ:001-theta}), and 
\[
	T^{\circ}(x) = \begin{cases}
		T(x), & \text{if } \circ = +,\\
		\neg T(x), & \text{if } \circ = -.
	\end{cases}
\]

The structure $\mathcal{C}^{+}$ satisfies the sentence $\xi_{+}$: the desired element $x$ can be chosen as the root of some appended tree, which is isomorphic to $\mathcal{T}[\emptyset']$. On the other hand, one can show that $\mathcal{C}^{-} \not\models \xi_{+}$.

Indeed, towards a contradiction, assume that some $x_{-} \in \mathcal{C}^{-}$ has the described properties. Then $x_{-}$ is the root of some appended tree, which is isomorphic to $\mathcal{T}[W]$ for a set $W$ belonging to the class $\mathcal{I}$. Recall that there is a unique element $\ell \in W\setminus \emptyset'$. Furthermore, there is a unique tuple $\widehat{x},y_0,\dots,y_{\ell},z$ from $\mathcal{C}^{-}$ with the property $	R(x_{-},\widehat{x}) \,\&\, \theta_{\ell}(\widehat{x},y_0,\dots,y_{\ell},z)$. Roughly speaking, in the appended tree, we have $\widehat{x} = a_{\ell}$, $y_j = d_{\ell,j}$, and $z = c^{\text{of this tree}}$. Since $\ell \in W$, the element $b_{\ell}$ belongs to the tree, therefore,
\[
	\mathcal{C}^{-} \models \exists v ( R(\widehat{x},v) \,\&\, U(v) \,\&\, v\neq x_{-} ),
\]
which gives a contradiction.

In a similar way, one shows that $\mathcal{C}^{-}\models \xi_{-}$ and $\mathcal{C}^{+}\not\models \xi_{-}$. By Theorem~\ref{theorem:characterization learning}$(a)$, we deduce that our class $\mathfrak{K}$ is learnable.

\medskip

Now, towards a contradiction, assume that the class $\mathfrak{K}$ is learnable by a computable learner $M$. Without loss of generality, one may assume the following: if a copy of the structure $\mathcal{C}^{\circ}$ (where $\circ \in \{ +,- \}$) is given as input, then in the limit, $M$ will output the symbol $\circ$.

Consider the computable sequence $(\mathcal{C}_k)_{k\in\omega}$ built in the construction. Then we have the following:
\begin{itemize}
	\item If $k\in X$, then $\mathcal{C}_k \cong \mathcal{C}^{+}$. Thus, given the data about $\mathcal{C}_k$, the learner $M$ outputs $+$ in the limit.
	
	\item If $k\not\in X$, then $\mathcal{C}_k \cong \mathcal{C}^{-}$. Given $\mathcal{C}_k$, the limit output of $M$ is equal to $-$.
\end{itemize}
Since the learner $M$ is computable, we deduce that $X$ is a $\Delta^0_2$ set, which contradicts the choice of our $X$. Therefore, $\mathfrak{K}$ is not learnable in a computable fashion. Theorem~\ref{theo:lower-bound-finite-case} is proved. 
\qed
\end{proof}

\section{Conclusion}
Let us conclude by briefly mentioning two ways of extending the above research. 

First, one may say that an oracle $X$ is \emph{low for learning structures} if any family $\mathfrak{K}$ which is learnable by an $X$-computable learner can already be learned computably.  A number of lowness notions have been investigated in computability theory (e.g., an oracle can be \emph{low for isomorphism}~\cite{franklin2014degrees}, \emph{low for bi-embeddability}~\cite{bazhenov2019degrees}, \emph{low for randomness}~\cite{nies2005lowness}, etc.). Yet, the study of oracles which do not supplement the learning power of $\mathbf{0}$  (called \emph{trivial}) is also an important tradition in algorithmic learning theory: Slaman and Solovay~\cite{slaman1991oracles} proved that, in the case of learning recursive functions, noncomputable trivial oracles coincide with the $1$-generic sets below  $\emptyset'$. So, one could try to characterize the class of oracles that are low for learning structures.

The second research direction that we want to suggest builds on the following observation: the results of this paper are all based on the fact that one could effectively recover (a copy of) a given structure $\mathcal{A}$ from knowing the conjecture $\ulcorner \A \urcorner$. What if one drops this assumption? For instance, suppose that a learner is allowed to output as conjectures \emph{any} index of the observed structure (according to some background enumeration of the computable structures). We would like to know whether the addition of multiple indices for a given structure extends the capabilities of computable learners (as is in the case of \emph{behaviourally correct} learning, see \cite{case1983comparison,jain1993non}).

\bibliographystyle{splncs04}
%\bibliography{Complexity-Refs}

\end{document}